\documentclass[11pt]{article}
\usepackage[colorlinks=false]{hyperref}
\usepackage{amsfonts, amsmath, amssymb, amsthm, gensymb, CJK, soul}
\usepackage[nolabel]{showlabels}
\usepackage[utf8]{inputenc}

\author{Zilin Jiang \begin{CJK}{UTF8}{gkai}姜子麟\end{CJK}\thanks{Department of Mathematics, Massachusetts Institute of Technology, Cambridge, MA 02139, USA. Email: {\tt
zilinj@mit.edu}. The work was done when the author was a postdoctoral fellow at Technion -- Israel Institute of Technology, and was supported in part by ISF grant nos 1162/15, 936/16.}}
\title{On spectral radii of unraveled balls}
\date{}

\newtheorem{theorem}{Theorem}
\newtheorem{corollary}[theorem]{Corollary}
\newtheorem{lemma}[theorem]{Lemma}

\theoremstyle{definition}
\newtheorem{definition}{Definition}

\theoremstyle{remark}
\newtheorem{remark}{Remark}

\newcommand{\from}{\colon}
\newcommand{\abs}[1]{\left\lvert {#1}\right\rvert}
\newcommand{\ip}[2]{\langle#1, #2\rangle}
\newcommand{\dset}[2]{\left\{#1 : #2\right\}}
\newcommand{\sset}[1]{\left\{#1\right\}}

\newcommand{\la}{\lambda}
\newcommand{\si}{\sigma}
\renewcommand{\epsilon}{\varepsilon}

\newcommand{\R}{\mathbb{R}}
\newcommand{\N}{\mathbb{N}}

\newcommand{\E}[1]{\mathrm{E}\left[{#1}\right]}
\newcommand{\pr}[1]{\operatorname{Pr}\left({#1}\right)}
\newcommand{\cpr}[2]{\operatorname{Pr}\left(#1\mid#2\right)}
\newcommand{\tr}{\operatorname{tr}}

\newcommand{\tg}{\tilde{G}}

\begin{document}

\maketitle

\begin{abstract}
  Given a graph $G$, the unraveled ball of radius $r$ centered at a vertex $v$ is the ball of radius $r$ centered at $v$ in the universal cover of $G$. We prove a lower bound on the maximum spectral radius of unraveled balls of fixed radius, and we show, among other things, that if the average degree of $G$ after deleting any ball of radius $r$ is at least $d$ then its second largest eigenvalue is at least $2\sqrt{d-1}\cos(\frac{\pi}{r+1})$.
\end{abstract}

\section{Introduction}\label{intro}

The well-known result of Alon and Boppana~\cite{MR1124768} states that for every $d$-regular graph $G$ containing two edges at distance $\ge 2r$, the second largest eigenvalue, denoted by $\la_2(G)$, of the adjacency matrix of $G$ satisfies:
\[
  \la_2(G) \ge 2\left(1-\frac{1}{r}\right)\sqrt{d-1}+\frac{1}{r}.
\]
Subsequently, Friedman~\cite[Corollary 3.6]{MR1208809} improved the above bound (see also \cite{MR2056091}): for every $d$-regular graph $G$ with diameter $\ge 2r$, $$\la_2(G) \ge 2\left(1-\frac{\pi^2}{2r^2}+O\left(r^{-4}\right)\right)\sqrt{d-1}.$$

All these proofs of the Alon--Boppana bound primarily relied on estimating the spectral radius of an induced subgraph on the vertices within certain distance from a given vertex or edge.

\begin{definition}
  Given a graph $G$ and a vertex $v$, a ball of radius $r$ centered at $v$, denoted by $G(v, r)$, is the induced subgraph of $G$ on the vertices within distance $r$ from $v$.
\end{definition}

The proof of Friedman uses the fact~\cite[Lemma 3.3]{MR1208809} that the spectral radius of $G(v, r)$ is at least the spectral radius of the $d$-regular tree of depth $r$. Note that the universal cover of a $d$-regular graph is the infinite $d$-regular tree. This motivates the following definition.

\begin{definition}
  Given a graph $G$, a walk $(v_0, v_1, \dots)$ on $G$ is non-backtracking if $v_i \neq v_{i+2}$ for all $i$. For every vertex $v$, define the tree $\tg(v, r)$ as follows: the vertex set consists of all the non-backtracking walks on $G$ of length $\le r$ starting at $v$, and two vertices are adjacent if one is a simple extension of another. In other words, $\tg(v, r)$ is the ball of radius $r$ centered at $v$ in the universal cover $\tg$ of $G$.
\end{definition}

We call $\tg(v, r)$ the \emph{unraveled ball} of radius $r$ centered at $v$, and we prove the following theorem on the spectral radii of unraveled balls. From now on, $\la_1(\cdot)$ denotes the spectral radius of a graph and $d(u)$ denotes the degree of $u$ in $G$.

\begin{theorem}\label{lb1}
  For any graph $G = (V, E)$ of minimum degree $\ge 2$ and $r\in \N$, there exists a vertex $v\in V$ such that \[
    \la_1(\tg(v, r)) \ge \frac{1}{\abs{E}} \sum_{u\in V}d(u)\sqrt{d(u)-1} \cdot \cos\left(\frac{\pi}{r+2}\right).
  \]
\end{theorem}

After presenting the proof of Theorem~\ref{lb1} in Section~\ref{sec_proofs}, we show in Section~\ref{sec_cover} a cheap lower bound on the spectral radius of the universal cover of a graph. We proceed in Section~\ref{sec_balls} and \ref{sec_second} to describe additional applications including an improvement to a result of Hoory~\cite{MR2102266}. The final section briefly discusses a potential extension to weighted graphs and its connection to the normalized Laplacian.

\section{Proof of Theorem~\ref{lb1}}\label{sec_proofs}

  The proof uses the old idea of constructing a test function by looking at non-backtracking walks (see e.g. \cite{MR3558041} and \cite{MR3775873}). The innovation here is to weight the test function using the eigenvector of a path.

\begin{proof}[Proof of Theorem~\ref{lb1}]
  Define $W_i$, for every $i \ge 1$, to be the set of all non-backtracking walks of length $i$ on $G$. Specifically $W_1$ is the set of directed edges of $G$. Define the forest $T$ as follows: the vertex set is $\bigcup_{i=1}^{r+1} W_i$ and two vertices are adjacent if and only if one is a simple extension of the other. For every $e = (v_0, v_1) \in W_1$, denote by $T_{e}$ the connected component of $T$ containing $e$. If one identifies every vertex $(v_0, v_1, \dots, v_i)$ in $T_e$, where $i \in [r+1]$, with the vertex $(v_1, \dots, v_i)$ in $\tg(v_1, r)$, then $T_e$ becomes a subgraph of $\tg(v_1, r)$. By the monotonicity of spectral radius, $\la_1(\tg(v_1, r)) \ge \la_1(T_e)$. Because $\la_1(T) = \max\dset{\la_1(T_{e})}{e\in W_1}$, there exists a vertex $v\in V$ such that $\la_1(\tg(v,r)) \ge \la_1(T)$. Set $\la = 2\cos(\tfrac{\pi}{r+2})$. It suffices to prove $$\la_1(T) \ge \la\cdot\sum_{u\in V}\frac{d(u)}{\abs{W_1}}\sqrt{d(u)-1}.$$
  
  Consider the following time-homogeneous Markov chain on $W_1$: the initial state $E_1$ is chosen uniformly at random from $W_1$, and given the current state $E_i = (v_{i-1}, v_i)$, the next state $E_{i+1}$ is chosen uniformly at random among $\dset{(v_i, v_{i+1})\in W_1}{v_{i+1}\neq v_{i-1}}$. Since the ending vertex of $E_i$ is always identical to the starting vertex of $E_{i+1}$ and the starting vertex of $E_i$ is always distinct from the ending vertex of $E_{i+1}$, we can join $E_1, E_2, \dots, E_i$ together to form a non-backtracking walk on $G$ of length $i$, which we denote by the random variables $Y_i = (X_0, X_1, \dots, X_i)$.
  
  Recall that $\la = 2\cos(\frac{\pi}{r+2})$ is the spectral radius of the path of length $r$. Let $(x_1, x_2, \dots, x_{r+1}) \in \R^{r+1}$ be an eigenvector of the path associated with $\la$. By the Rayleigh principle, we have
  \begin{equation} \label{p1eq1}
    \sum_{i=2}^{r+1}{2x_{i-1}x_i} = \la \cdot \sum_{i=1}^{r+1} x_i^2
  \end{equation}
  Define the vector $f\from \bigcup_{i=1}^{r+1}W_i \to \R$ by $f(w) = x_i\sqrt{\pr{Y_{i} = w}}$ for $w \in W_i$, and define the matrix $A$ to be the adjacency matrix of the forest $T$. For $w = (v_0, v_1, \dots, v_i)$, denote by $w^- = (v_0, v_1, \dots, v_{i-1})$. We observe that
  \begin{align}
    \ip{f}{f} & = \sum_{i=1}^{r+1}\sum_{w\in W_i}f(w)^2 = \sum_{i=1}^{r+1}\sum_{w\in W_i}x_i^2\pr{Y_i = w} = \sum_{i=1}^{r+1}x_i^2, \label{p1eq2}\\
    \ip{f}{Af} & = \sum_{i=2}^{r+1}\sum_{w\in W_i}2f(w^-)f(w) = \sum_{i=2}^{r+1}2x_{i-1}x_i\sum_{w\in W_i}\sqrt{\pr{Y_{i-1}=w^-}\pr{Y_i=w}}. \label{p1eq3}
  \end{align}
  
  By the Markov property, for every $i \ge 2$ and $w = (v_0, v_1, \dots, v_i) \in W_i$,
  $$\frac{\pr{Y_{i}=w}}{\pr{Y_{i-1}=w^-}} = \cpr{E_i = (v_{i-1}, v_i)}{E_{i-1} = (v_{i-2}, v_{i-1})} = \frac{1}{d(v_{i-1})-1}.$$
  Thus the inner summation in the right hand side of \eqref{p1eq3} equals
  \begin{equation}\label{p1eq4}
    \sum_{w = (v_0, v_1, \dots, v_i)\in W_i}\sqrt{d(v_{i-1})-1}\pr{Y_i=w} = \E{\sqrt{d(X_{i-1})-1}} = \sum_{v\in V}\sqrt{d(v)-1}\pr{X_{i-1}=v}.    
  \end{equation}
  Since the minimum degree of $G$ is $\ge 2$, the Markov chain has no absorbing states. Moreover, one can easily check that the uniform distribution on $W_1$ is a stationary distribution of the Markov chain, that is, $\pr{E_i = e} = 1/\abs{W_1}$ for all $i \ge 1$ and $e\in W_1$. Thus $\pr{X_{i-1} = v} = d(v)/\abs{W_1}$ for all $i \ge 2$ and $v\in V$. Plugging this into \eqref{p1eq4}, we can simplify \eqref{p1eq3} to \[
    \ip{f}{Af} = \sum_{i=2}^{r+1} 2x_{i-1}x_i\sum_{v\in V(G)}\frac{d(v)}{\abs{W_1}}\sqrt{d(v)-1}.
  \]

  Finally we combine with \eqref{p1eq1} and \eqref{p1eq2}, and the Rayleigh principle $\la_1(T) \ge {\ip{f}{Af}}/{\ip{f}{f}}$.
\end{proof}

\section{Spectral radius of the universal cover}\label{sec_cover}

Since $\tilde{G}(v, r)$ is an induced subgraph of $\tilde{G}$, the monotonicity of spectral radius implies immediately a lower bound on $\la_1(\tilde{G})$ by letting $r$ go to infinity in Theorem~\ref{lb1}.

\begin{corollary}
  For any graph $G = (V, E)$ of minimum degree $\ge 2$, the spectral radius of its universal cover satisfies 
  \begin{equation}\label{lb2}
    \la_1(\tg) \ge \frac{1}{\abs{E}}\sum_{u\in V}d(u)\sqrt{d(u)-1}.
  \end{equation}
\end{corollary}

\begin{remark}
By the inequality of arithmetic and geometric means, the right hand side of \eqref{lb2} satisfies
\[
  \frac{1}{\abs{E}}\sum_{u\in V}d(u)\sqrt{d(u)-1} = 2 \cdot \frac{\sum_{u\in V}d(u)\sqrt{d(u)-1}}{\sum_{u\in V}d(u)} \ge 2 \prod_{u\in V}\left(\sqrt{d(u)-1}\right)^{\frac{d(u)}{\sum_{v\in V}d(v)}},
\] which recovers the lower bound on $\la_1(\tg)$ in \cite[Theorem 1]{MR2102266}.  
\end{remark}

\section{Maximum spectral radius of balls}\label{sec_balls}

The following result, which is essentially due to Mohar~\cite[Theorem 2.2]{MR2679612}, connects the spectral radii of a ball and its corresponding unraveled ball.

\begin{lemma}\label{unravel}
  For every vertex $v$ of a graph $G$ and $r\in\N$, $\la_1(G(v, r)) \ge \la_1(\tg(v, r))$.
\end{lemma}

To prove Lemma~\ref{unravel} we need the following simple fact. For the sake of completeness we include the short proof in Appendix~\ref{appendix}.

\begin{lemma}\label{closed_walks}
  For every connected graph $G = (V, E)$ and every vertex $v \in V$, $\la_1(G) = \limsup \sqrt[k]{s_k(v)}$, where $s_k(v)$ is the number of closed walks of length $k$ starting at $v$ in $G$. In fact, $\la_1(G) = \lim \sqrt[2k]{s_{2k}(v)}$.
\end{lemma}

\begin{proof}[Proof of Lemma~\ref{unravel}]
  Recall that a vertex of $\tg(v, r)$ is a non-backtracking walk of length $\le r$ starting at $v$. Denote the non-backtracking walk of length $0$ starting at $v$ by $w := (v)$. For every $k$, we naturally map a closed walk $w = w_1, w_2, \dots, w_k = w$ of length $k$ in $\tg(v,r)$ to a closed walk $v = v_1, v_2, \dots v_k = v$ of length $k$ in $G(v,r)$, where $v_j$ is the terminal vertex of $w_j$ for $j\in[k]$. One can show that this map is injective, and so the number of closed walks of length $k$ starting at $v$ in $G(v,r)$ is at least the number of closed walks of length $k$ starting at $w$ in $\tg(v,r)$. Lemma~\ref{closed_walks} thus implies that $\la_1(G(v,r)) \ge \la_1(\tg(v,r))$.
\end{proof}

We shall combine Lemma~\ref{unravel} and Theorem~\ref{lb1} to provide a lower bound on the maximum spectral radius of balls in Lemma~\ref{lb3}, which slightly strengthens \cite[Lemma 12]{jp17}. We need the following fact.

\begin{theorem}[Theorem 2 of Collatz and Sinogowitz~\cite{MR0087952}; Theorem 3 of Lov\'asz and Pelik\'an~\cite{MR0416964}]\label{lp}
  If $G$ is a tree of $n$ vertices then $\la_1(P_n) \le \la_1(G)$, where $P_n$ is the path with $n$ vertices.
\end{theorem}

\begin{lemma}\label{lb3}
  For any graph $G = (V, E)$ of average degree $d \ge 1$ and $r\in \N$, there exists $v\in V$ such that $\la_1(G(v, r)) \ge 2\sqrt{d-1}\cos(\tfrac{\pi}{r+2})$.
\end{lemma}

\begin{proof}
  If $G$ has more than one connected component, we shall just prove for one of the connected components with average degree $\ge d$. Hereafter, we assume that $G$ is connected.

  \textbf{Case $1 \le d < 2$:} For a connected graph, having an average degree $< 2$ is the same as being a tree. Pick any $v\in V$. If $G(v, r) = G$, then $\la_1(G(v, r)) = \la_1(G) \ge d = (d - 1) + 1 \ge 2\sqrt{d-1}$. Otherwise $G(v, r)$ is a tree of $\ge r+1$ vertices, and $\la_1(G(v, r)) \ge \la_1(P_{r+1}) = 2\cos(\tfrac{\pi}{r+2})$ by Theorem~\ref{lp}.
  
  \textbf{Case $d \ge 2$:} Since removing leaf vertices from a graph of average degree $d \ge 2$ cannot decrease its average degree, without loss of generality, we may assume that the minimum degree of $G$ is $\ge 2$. By Lemma~\ref{unravel} and Theorem~\ref{lb1}, there exists a vertex $v\in V$ such that $$\la_1(G(v, r)) \ge \la_1(\tg(v,r)) \ge \frac{1}{\abs{E}}\sum_{u\in V}d(u)\sqrt{d(u)-1}\cdot \cos\left(\frac{\pi}{r+2}\right).$$ A straightforward calculation can verify that the function $x\mapsto x\sqrt{x-1}$ is convex for $x\ge 2$. It follows from the Jensen's inequality that the right hand side of the above is at least 
  \[
    \frac{1}{\abs{E}} \cdot \abs{V} d\sqrt{d-1}\cdot \cos\left(\frac{\pi}{r+2}\right) = 2\sqrt{d-1}\cos\left(\frac{\pi}{r+2}\right). \qedhere    
  \]
\end{proof}

\section{Second largest eigenvalue}\label{sec_second}

It is natural to generalize the Alon--Boppana bound to graphs that may not be regular. It is conceivable that for any sequence of graphs $G_i$ with average degree $\ge d$ and growing diameter, $\liminf \la_2(G_i) \ge 2\sqrt{d-1}$. However, Hoory constructed in \cite{MR2102266} a counterexample to such a statement. In his construction, the average degree drops drastically after deleting a ball of radius $1$. Hoory then extended the Alon--Boppana bound to graphs that have a robust average degree.

\begin{definition}
  A graph has an $r$-robust average degree $\ge d$ if the average degree of the graph is $\ge d$ after deleting any ball of radius $r$.
\end{definition}

\begin{theorem}[Theorem 3 of Hoory \cite{MR2102266}]\label{hoory}
  Given a real number $d \ge 2$ and a natural number $r \ge 2$, for any graph $G$ that has an $r$-robust average degree $\ge d$, its second largest eigenvalue in absolute value satisfies:
  \begin{equation} \label{hoory_la2}
    \max\sset{\la_2(G), \la_{-1}(G)} \ge 2\left(1 - c \cdot \frac{\log r}{r}\right)\sqrt{d-1},
  \end{equation}
  where $\la_{-1}(G)$ denotes the smallest eigenvalue of $G$, and $c$ is an absolute constant.
\end{theorem}


It is noticeable that Theorem~\ref{hoory} may not be optimal in comparison to the Alon--Boppana bound. The left hand side of \eqref{hoory_la2} should be simply $\la_2(G)$, and inside the right hand side $c\cdot \frac{\log r}{r}$ could be improved to $c \cdot \frac{1}{r^2}$. We prove that this is indeed the case.

\begin{theorem}\label{ab_robust}
  Given a real number $d \ge 1$ and a natural number $r \ge 1$, if a graph $G$ has an $r$-robust average degree $\ge d$, then
  \[
    \la_2(G) \ge 2\sqrt{d-1}\cos\left(\frac{\pi}{r+1}\right).
  \]
\end{theorem}

\begin{proof}
  After deleting an arbitrary ball of radius $r$, as the average degree is $\ge d$, by Lemma~\ref{lb3} and the monotonicity of spectral radius, there is $v_1\in V$ such that the spectral radius of $G_1 := G(v_1, r-1))$ is at least $2\sqrt{d-1}\cos(\tfrac{\pi}{r+1}) =: \la_*$. Let $G' = (V', E')$ be the graph after deleting the ball of radius $r$ centered at $v_1$ from $G$. Repeating this argument, we can find $v_2 \in V'$ such that the spectra radius of $G_2 := G'(v_2, r-1)$ is at least $\la_*$. For $i=1,2$, let $A_i$ be the adjacency matrix of $G_i$ and let $f_i$ be the eigenvector of $A_i$ associated with $\la_1(G_i)$.
  
  Denote the adjacency matrix of $G$ by $A$. Choose scalars $c_1, c_2$, not all zero, such that the vector $f\from V\to \R$, defined by
  \[
    f(u) = \begin{cases}
      c_1f_1(u) & \text{if } u\in V(G_1);\\
      c_2f_2(u) & \text{if } u\in V(G_2);\\
      0      & \text{otherwise},
    \end{cases}
  \] is perpendicular to an eigenvector of $A$ associated with $\la_1(G)$. If $u\in V(G_1)$, that is $u$ is within distance $r-1$ from $v_1$, and $v$ is adjacent to $u$, then $v$ is within distance $r$ from $v_1$, hence $v \notin V(G_2)$. In other words, $\sset{u,v}\notin E$ for all $u\in V(G_1)$ and $v\in V(G_2)$. Thus we obtain a Rayleigh quotient from which $\la_2(G)\ge\la_*$ follows: \[
    \frac{\ip{f}{Af}}{\ip{f}{f}} = \frac{c_1^2\ip{f_1}{A_1f_1} + c_2^2\ip{f_2}{A_2f_2}}{c_1^2\ip{f_1}{f_1}+c_2^2\ip{f_2}{f_2}} = \frac{c_1^2\la_1(G_1)\ip{f_1}{f_1} + c_2^2\la_1(G_2)\ip{f_1}{f_1}}{c_1^2\ip{f_1}{f_1}+c_2^2\ip{f_2}{f_2}} \ge \la_*. \qedhere
  \]
\end{proof}

\section{Concluding remarks} \label{sec_remarks}

The \emph{normalized Laplacian} $N$ of $G = (V, E)$ is defined by $N = I - D^{-\frac{1}{2}}AD^{-\frac{1}{2}}$, where $D$ is the diagonal degree matrix with $D_{v,v} = d(v)$ for all $v\in V$ and $A$ is the adjacency matrix. The second smallest eigenvalue of $N$, denoted by $\mu_2(G)$, is tightly connected to various expansion properties of $G$ (see, for example, \cite[Section 2]{MR3558041}).

In the context of the normalized Laplacian, the Alon--Boppana bound says that for a $d$-regular graph $G$ with diameter $\ge 2r$, $\mu_2(G) \le 1 - \frac{2\sqrt{d-1}}{d}\left(1-\frac{\pi^2}{2r^2}+O(r^{-4})\right)$. Young~\cite[Section 3]{y11} refuted the natural generalization by showing an infinite family of graphs $G_1, G_2, \dots$ with common average degree $d$ and growing diameter and some fixed $\epsilon > 0$ such that $\mu_2(G_i) \ge 1 - \frac{2\sqrt{d-1}}{d} + \epsilon$ for all $i$. He also proved an upper bound of the form $\mu_2(G) \le 1-\frac{2\sqrt{d-1}}{\tilde{d}}\left(1-c\cdot\frac{\ln k}{k}\right)$, where $d$ is the average degree, $\tilde{d}$ is the second order average degree and $k$ is the normalized Laplacian eigenradius (see \cite[Theorem 6]{y11}). Recently Chung proved, under some technical assumptions
on $G$, another upper bound of the form $\mu_2(G)\le 1 - \si(G)(1-\frac{c}{k})$, where $\si(G) := {2\sum_{u\in V}d(u)\sqrt{d(u)-1}}/{\sum_{u\in V}d(u)^2}$ and $k$ is the diameter of $G$ (see \cite[Theorem~9]{MR3558041}).

Observe that the matrix $D^{-\frac{1}{2}}AD^{-\frac{1}{2}}$ in the definition of $N$ can be seen as a graph $G$ with the weight $d(u)^{-\frac{1}{2}}d(v)^{-\frac{1}{2}}$ assigned to each edge $\sset{u,v}$. Moreover, the second largest eigenvalue of this weighted graph is equal to $1 - \mu_2(G)$. Based on these two observations, the author believes that the machinery developed in this paper can be generalized to weighted graph to provide a better upper bound on $\mu_2(G)$, possibly under fewer assumptions on the graph.

\section*{Acknowledgements}

Some ideas of this paper have their origins in the recent joint work with Polyanskii on equiangular lines in Euclidean spaces (see \cite[Lemma 12]{jp17}). Besides, the author thanks the referees for many improvements in presentation.

\bibliographystyle{jalpha}
\bibliography{local_spectral_radius}

\appendix

\section{Spectral radius and closed walks}\label{appendix}

\begin{proof}[Proof of Lemma~\ref{closed_walks}]
  Let $v = v_1, v_2, \dots, v_n$ be the vertices of $G$, and let $\la_1(G) = \la_1 \ge \la_2 \ge \dots \ge \la_n$ be the eigenvalues of the adjacency matrix $A$ of $G$. An elementary graph theoretic interpretation identifies the trace of $A^k$ as the number of closed walks of length $k$ in $G$. But a standard matrix result equates $\tr A^k$ to the $k$th moment of $A$ defined as $\sum_{i=1}^n \la_i^k$. Thus, we have found the following:
  \begin{equation}\label{moment}
    \sum_{i=1}^n \la_i^k = \tr A^k = \sum_{i=1}^n s_k(v_i) =: s_k.
  \end{equation}
  Observe that $s_k$ is always a natural number. We see from \eqref{moment} that $\la_1 \ge \abs{\la_i(G)}$ for all $i\in[n]$, hence 
  \begin{equation}\label{limsup}
    \la_1(G) = \limsup \sqrt[k]{s_k}.
  \end{equation}
  For every $i\in[n]$, let $k_i$ be the distance from $v$ to $v_i$. By prepending the walk from $v$ to $v_i$ of length $k_i$ and appending the reverse, we extend a closed walk of length $k$ starting at $v_i$ to one of length $k+2k_i$ starting at $v$, and we obtain that $s_{k}(v_i) \le s_{k+2k_i}(v)$. Similarly by appending a closed walk of length $2$ starting at $v$, we extend a closed walk of length $k$ starting at $v$ to one of length $k+2$, and we obtain that $s_k(v) \le s_{k+2}(v)$ for all $k\in\N$. Thus $$s_k(v) \le s_k = \sum_{i=1}^n s_k(v_i) \le \sum_{i=1}^n s_{k+2k_i}(v) \le n \cdot s_{k+2k^*}(v),$$ where $k^* = \max\sset{k_1, k_2, \dots, k_n}$. In view of \eqref{limsup}, we get that $\la_1(G) = \limsup\sqrt[k]{s_k(v)}$. Lastly, note that \eqref{limsup} can be made more precise as $\la_1(G) = \lim \sqrt[2k]{s_{2k}}$ to obtain $\la_1(G) = \lim \sqrt[2k]{s_{2k}(v)}$.
\end{proof}

\end{document}